\documentclass[12pt]{amsart}
\usepackage{amsfonts}
\usepackage{amssymb}
\usepackage{color}
\usepackage{dsfont}
\usepackage{setspace}
\numberwithin{equation}{section}
\theoremstyle{plain}
 \newtheorem{theorem}{Theorem}[section]
 
 \newtheorem{corollary}[theorem]{Corollary}
 \newtheorem{proposition}[theorem]{Proposition}
 
 \newtheorem{assumption}[theorem]{Assumption}
 
\theoremstyle{definition}
 \newtheorem{definition}[theorem]{Definition}
 \newtheorem{example}[theorem]{Example}
 \newtheorem{remark}[theorem]{Remark}

\newcommand{\bN}{\mathbb{N}}
\newcommand{\bZ}{\mathbb{Z}}
\newcommand{\bR}{\mathbb{R}}

\newcommand{\bC}{\mathbb{C}}

\newcommand{\supp}{\mbox{\rm supp }}

\newcommand{\bE}{\mathbb{E}}

\newcommand{\one}{\mathbf{1}}

\allowdisplaybreaks

\begin{document}

\vspace{5mm}
\begin{center}
{\bf
{\large
L\'{e}vy driven CARMA generalized processes and stochastic partial differential equations}}

\vspace{5mm}

David Berger\\
\end{center}
\vspace{5mm}
We give a new definition of a L\'{e}vy driven CARMA random field, defining it as a generalized solution of a stochastic partial differential equation (SPDE). Furthermore, we give sufficient conditions for the existence of a mild solution of our SPDE. Our model finds a connection between all known definitions of CARMA random fields, and especially for dimension 1 we obtain the classical CARMA process. 
\section{Introduction}
Autoregressive moving average (ARMA) processes are very well known processes in time series analysis. An ARMA$(p,q)$ process $(X_k)_{k\in\bZ}$, $p,q\in\bN_0$, is given by
\begin{align}\label{eqarma}
X_k-\sum\limits_{i=1}^p a_iX_{k-i}=W_k+\sum\limits_{j=1}^q b_j W_{k-j},
\end{align} 
where $a_1,\dotso,a_p,b_1,\dotso,b_q\in\bC$ are deterministic coefficients and $(W_k)_{k\in\bZ}$ is white noise or even an independent and identically distributed (iid) sequence of random variables. In short form we can also write
\begin{align*}
a(B)X_k=b(B)W_k,
\end{align*}
where $a(z)=1-\sum\limits_{i=1}^pa_i z^k$, $b(z)=1+\sum\limits_{j=1}^q b_j z^j$ are polynomials and $B$ is the shift operator defined by $B^l Y_k=Y_{k-l}$ for $l\in\bN$. ARMA$(p,q)$ processes were generalized in various ways and have many applications, e.g. in finance, astrophysics, engineering and traffic data, see [\ref{M}], [\ref{S}], [\ref{Zvaritch}] and [\ref{Klepsch}].\\
As the solution of $(\ref{eqarma})$ is a discrete process on a lattice, a possible way to generalize the concept is to study a continous version of $(\ref{eqarma})$, which is called continuous ARMA (CARMA) process. A CARMA$(p,q)$ process $(X_t)_{t\in\bR}$, where $p>q$, is given by
\begin{align}\label{eq1.1}
X_t=b'Y_t,\,t\in\bR,
\end{align}
where $Y=(Y_t)_{t\in\bR}$ is a $\bC^p$-valued process satisfying the stochastic differential equation
\begin{align}\label{eq1.2}
dY_t=AY_tdt+e_pdL_t
\end{align}
with
\begin{align*}
A=\begin{pmatrix}
0&1&0&\dotso&0\\
0&0&1&\dotso&0\\
\vdots&\vdots&\vdots&\ddots&\vdots\\
0&0&0&\dotso&1\\
-a_p&-a_{p-1}&-a_{p-2}&\dotso&-a_1
\end{pmatrix}, 
e_p=\begin{pmatrix}0\\
0\\
\vdots\\
0\\
1
\end{pmatrix}\in\mathbb{C}^p\textrm{ and }
b=\begin{pmatrix}b_0\\
b_1\\
\vdots\\
b_{p-2}\\
b_{p-1}
\end{pmatrix},
\end{align*}
where $a_1,\dotso,a_p,b_0,\dotso,b_{p-1}\in\bC$ are determinstic coefficients such that $b_q\neq 0$ and $b_j=0$ for every $j>q$, $b'$ denotes the transpose of $b$ and $L=(L_t)_{t\in\bR}$ is a two-sided L\'{e}vy process. The equations $(\ref{eq1.1})$ and $(\ref{eq1.2})$ are the so called state-space representation of the formal stochastic differential equation
\begin{align*}
a(D)Y_t=b(D)DL_t,
\end{align*} 
with $D$ the differential operator and $a(z)=z^p+a_1z^{p-1}+\dotso+a_p$ and $b(z)=b_0+b_1z+\dotso+b_qz^{q}$ are polynomials. In [\ref{Lindner}] necessary and sufficient conditions on $L$ and $A$ were given such that there exists a strictly stationary solution of $(\ref{eq1.1})$ and $(\ref{eq1.2})$, namely it was shown that it is sufficient and necessary that $\bE \log^+(|L_1|)<\infty$. CARMA processes have many applications, see [\ref{Garcia}] and [\ref{Brockwell3}].\\
As the CARMA process is defined on $\bR$, spatial problems cannot be easily transferred. As a consequence, there are some extensions of the CARMA process to the multidimensional setting. Lately, there were the two papers of Brockwell and Matsuda [\ref{Brockwell}] and Pham [\ref{Pham}], who introduce different concepts of CARMA processes in the multidimensional setting. In [\ref{Brockwell}] the new CARMA random field was given by
\begin{align}\label{eqbrockwell}
S_d(t):=\int\limits_{\bR^d} \sum\limits_{r=1}^p \frac{b(\lambda_r)}{a'(\lambda_r)}e^{\lambda_r\|t-u\|}dL(u),
\end{align}
where $dL$ denotes the integration over a L\'{e}vy bases, $a$ and $b$ are polynomials such that $a(z)=\prod_{i=1}^p(z^2-\lambda_i^2)$ and some further restrictions. The model has a well understood second order behaviour and can be used for statistical estimation. However, the authors do not deal with a dynamical description.\\
Pham [\ref{Pham}] follows another way and defines a CARMA random field $Y$ as a mild solution of the system of SPDEs given by
\begin{align}\label{eqpham}
Y(t)&=b'X(t),\,t\in\bR^d,\\
(I_p\partial_d-A_d)\cdots(I_p\partial_1-A_1)X(t)&=c\dot{L}(t),\,t\in\bR^d,
\end{align}
where $\dot{L}$ is a L\'{e}vy basis, $A_1,\dotso,A_d\in \bR^{p\times p}$ are matrices and $I_p$ is the identity matrix. Pham speaks of causal CARMA random fields, as the solution of the system (\ref{eqpham}) depends only on the past in the sense that the solution at point $x$ depends solely on the behavior of $\dot{L}$ on $(-\infty,x_1]\times\cdots\times (-\infty,x_d]$. So we can see directly that there is a big difference between these two definitions.\\
 The aim of this paper is to find a connection between these two models and give a generalized definition of CARMA random fields. Our starting point is the equation
\begin{align}\label{eqberger}
p(D)s=q(D)\dot{L},
\end{align}
where $p,q$ are polynomials in $d$ variables, $D$ denotes the differential operator and $\dot{L}$ denotes L\'{e}vy white noise. Our solution $s$ is defined as a generalized solution, see Section 3. We will start with an abstract analysis of this problem and prove for a far more general class then $(\ref{eqberger})$ the existence of a generalized solution under relatively mild conditions on the L\'{e}vy white noise. Our solution is similar to the definition of generalized CARMA$(p,q)$ process in [\ref{Hannig}] and as there, we do not assume that the degree of the polynomial $p$ is higher than the degree of the polynomial $q$. We will discuss two examples, which are related to the processes of Brockwell and Matsuda [\ref{Brockwell}] and Pham [\ref{Pham}]. We will also give certain conditions on $p$ and $q$ that guarantee that the obtained generalized solutions are random fields.\\
The above mentioned results can be found in Section \ref{section1} and Section \ref{section2}, where our main results are Theorem \ref{theorem1} and Theorem \ref{theoremcarma}. In Section \ref{section0} we recall some basic notation. In Section \ref{section1} we recall the definitions of L\'{e}vy white noise and generalized random processes. Moreover, we prove that a convolution operator with certain properties regarding his integrability defines a generalized random process and as an application we will study stochastic homogeneous elliptic partial differential equations. In Section \ref{section2} we use this theorem to show the existence of our CARMA generalized processes. Moreover, we study the concept of mild solutions in Section \ref{sectionCRF}, prove existence of mild CARMA random fields and show some connections between the mild and generalized solutions. In Section \ref{section3} we study the moment properties of our CARMA random fields and show that if the L\'{e}vy white noise has existing $\alpha$-moment for some $0<\alpha\le 2$, then the CARMA random field has also finite $\alpha$-moment, see Proposition \ref{propositionmoment}. In Section \ref{section8} we will study the connection between our model and the CARMA random field of Brockwell and Matsuda [\ref{Brockwell}]. 
\section{Notation and Preliminaries}\label{section0}
To fix notation, by $(\Omega,\mathcal{F})$ we denote a measurable space, where $\Omega$ is a set and $\mathcal{F}$ is a $\sigma$-algebra and by $L^0(\Omega,\mathcal{F},\mathbb{K})$ we denote all measurable functions $f:\Omega\to\mathbb{K}$ with respect to $\mathcal{F}$ where $\mathbb{K}=\bR, \bC$. In the case that $\mathcal{F}$ and $\mathbb{K}$ are clear from the context we set $L^0(\Omega)=L^0(\Omega,\mathcal{F},\mathbb{K})$. If we consider a probability space $(\Omega,\mathcal{F}, \mathcal{P})$, where $\mathcal{P}$ is a probability measure on $(\Omega,\mathcal{F})$, we say that a sequence $(f_n)_{n\in\bN}\subset L^0(\Omega)$ converges to $f$ in $L^0(\Omega)$ if $f_n$ converges in probability to $f$ with respect to the measure $\mathcal{P}$. In the case of $(\bR^d,\mathcal{B}(\bR^d))$ we denote by $\mathcal{B}(\bR^d)$ the Borel-$\sigma$-set on $\bR^d$. $\mathcal{B}_b(\bR^d)$ is the set of all Borel sets, which are bounded. \\
We write $\bN=\{1,2,\dotso\}$, $\bN_0=\bN\cup \{0\}$ and $\bZ,\,\bR,\,\bC$ for the set of integers, real numbers and complex numbers, respectively. If $z\in \bC$, we denote by $\Im z$ and $\Re z$ the imaginary and the real part of $z$. $\|\cdot\|$ denotes the Euclidean norm and $r^+:=\max\{0,r\}$ for every $r\in\bR$ . The indicator function of a set $A\subset \bR^d$, $d\in\bN$, is denoted by $\one_A$. By $L^p(\bR^d, A)$ for $A\subseteq \bC$ and $0<p\le \infty$ we denote the set of all Borel-measurable functions $f:\bR^d \to A$ such that $\int_{\bR^d} |f(x)|^p\,\lambda^d(dx)<\infty$ for $0<p<\infty$ and $\mathop{\textrm{ess sup}}_{x\in\bR^d}|f(x)|<\infty$ for $p=\infty$, where $\lambda^d$ is the $d-$dimensional Lebesgue measure. We denote by $||f||_{L^p}=\left(\int_\bR |f(x)|^p\,\lambda(dx)\right)^{1/p}$ for $0<p<\infty$ and $\|f\|_{L^\infty}=\mathop{\textrm{ess sup}}_{\bR^d}|f|$ the $L^p$-(quasi-)norm for a measurable function $f$. By $d_f$ we denote the distribution function of $f$, which means that 
\begin{align}
d_f(\alpha):=\lambda^d(\{x\in \bR^d: |f(x)|>\alpha\}),\, \alpha\ge 0.
\end{align}
We denote by $B_R(x)$ the set $\{y\in \bR^d: \|x-y\|<R\}$ and $x\wedge y:=\min\{x,y\}$ for two real numbers $x$ and $y$. For a set $A\subset \bR^d$ and an element $x\in\bR^d$ we set $dist(x,A):=\inf\{\|x-y\|:y\in A\}$. 
The space $\mathcal{D}(\bR^d)$ denotes the set of all infinitely differentiable functions $f:\bR^d\to\bR$ with compact support, where we denote the support of $f$ by $\supp f$. The topological dual space of $\mathcal{D}(\bR^d)$ will be denoted by $\mathcal{D}'(\bR^d)$, where an element  $u\in\mathcal{D}'(\bR^d)$ is called a distribution. We will write $\langle u,\varphi\rangle:=u(\varphi)$ for $\varphi\in\mathcal{D}(\bR^d)$. We say that a function $a:Y\to\bR$ from some function space $Y$ acts as a Fourier multiplier for some function space $X$ to a function space $R$ with well-defined Fourier transform $\mathcal{F}$ if $a:X\to R$ is defined by $a(u):=\mathcal{F}^{-1}(a´\mathcal{F}u)$, where $(a\mathcal{F}(u))(t)=a(t)\mathcal{F}(u)(t)$ such that the inverse Fourier transform $\mathcal{F}^{-1}$ is well-defined. For a function $f\in L^1(\bR^d,\bC^d)$ we set $\mathcal{F} f(x)=\int\limits_{\bR^d} e^{-i \langle z, x\rangle}f(z)\lambda^d(dz)$ and the $L^2$-Fourier transform likewise. Let $p(z)=\sum\limits_{|\alpha|\le m}p_{\alpha}z^{\alpha}$, $\alpha\in \bN^d_0$ and $z^{\alpha}=z_1^{\alpha_1}\dotso z_d^{\alpha_d}$, such that $p_{\beta}\neq 0$ for some $\beta$ with $|\beta|:=\beta_1+\dotso+\beta_d=m$. Then we define $\textrm{deg}(p):=m$, the defree of $p$. We set $D^{\alpha}=\partial_{x_1}^{\alpha_1}\dotso\partial_{x_d}^{\alpha_d}$ for $\alpha \in \bN^d_0$. We denote by $A^*$ the adjoint of the operator $A$.\\
We recall here the definition of a L\'{e}vy basis, as we explain some connection between a L\'{e}vy basis and generalized stochastic process, which will be defined later.

\begin{definition}[see {[\ref{Rajput}, p. 455]}]\label{randommeasures}
A \emph{L\'{e}vy basis} is family $(L(A))_{A\in\mathcal{B}_b({\bR^d})}$ of real valued random variables such that
\begin{itemize}
\item[i)]$L(\bigcup_{n=0}^\infty A_n)=\sum_{n=0}^\infty L(A_n)$ $a.s.$ for pairwise disjoint sets $(A_n)_{n\in\bN_0}\subset \mathcal{B}_b(\bR^d)$ with $\bigcup_{n\in\bN_0}A_n\in \mathcal{B}_b(\bR^d)$,
\item[ii)] $L(A_i)$ are independent for pairwise disjoint sets $A_1,\dotso,A_n\in \mathcal{B}_b(\bR^d)$ for every $n\in\bN$,
\item[iii)] there exist $a\in [0,\infty)$, $\gamma \in \bR$ and a L\'{e}vy measure $\nu$ on $\bR$ (i.e. a measure $\nu$ on $\bR$ such that $\nu(\{0\})=0$ and $\int\limits_{\bR} \min\{1,x^2\}\nu(dx)<\infty$) such that
\begin{align*}
\bE e^{iz L(A)}=\exp\left(\psi(z)\lambda^d(A)\right)
\end{align*}
for every $A\in \mathcal{B}_b(\bR^d)$, where
\begin{align*}
\psi(z):=i\gamma z-\frac{1}{2}az^2+\int\limits_{\bR} (e^{ixz}-1-ixz\one_{[-1,1]}(x))\nu(dx),\quad z\in \bR.
\end{align*}
The triplet $(a,\gamma, \nu)$ is called the \emph{characteristic triplet} of $L$ and $\psi$ its \emph{characteristic exponent}. By the L\'{e}vy-Khintchine formula, $L(A)$ is then infinitely divisible. 
\end{itemize}
\end{definition}
\vspace{5mm}

\section{SPDEs and generalized solutions}\label{section1}
\subsection{The concept of generalized solutions}
This section deals with L\'{e}vy white noise and the definition of solutions of the SPDEs given in (\ref{eqberger}). We will prove a multiplier theorem for general L\'{e}vy white noise and use this theorem to prove the existence of our CARMA random process. We will follow mainly [\ref{Fageot}, Section 2].\\
As already mentioned, we denote by $\mathcal{D}(\bR^d)$ the space of infinitely differentiable functions with compact support, where we assume that the space is equipped with the usual topology, i.e. we say that a sequence $(\varphi_n)_{n\in\bN}\subset \mathcal{D}(\bR^d)$ converges to $\varphi$ in $\mathcal{D}(\bR^d)$ if there exists a compact subset $K\subset \bR^d$ such that $\supp\varphi_n,\supp \varphi\subset K$ for every $n\in\bN$ and $\sup_{x\in\bR^d}|D^{\alpha}(\varphi_n(x)-\varphi(x))|\to 0$ for $n\to\infty$ for every multiindex $\alpha\in\bN^d_0$.\\
Let $(\Omega,\mathcal{F},\mathcal{P})$ be a probability space. We recall the definition of a generalized random process.
\begin{definition}[see {[\ref{Fageot}, Definition 2.1]}]
A generalized random process is a linear and continuous function $s:\mathcal{D}(\bR^d)\to L^0(\Omega)$. The linearity means that, for every $\varphi_1,\varphi_2\in\mathcal{D}(\bR^d)$ and $\gamma\in\bR$,
$$s(\varphi_1+\gamma \varphi_2)=s(\varphi_1)+\gamma s(\varphi_2) \textrm{ almost surely}. $$
The continuity means that if $\varphi_n\to\varphi$ in $\mathcal{D}(\bR^d)$, then $s(\varphi_n)\to s(\varphi)$ in $L^0(\Omega)$. 
\end{definition}
As shown in [\ref{Walsh}, Corollary 4.2], there exists a measurable version from $(\Omega,\mathcal{F})$ to $(\mathcal{D}'(\bR^d),\mathcal{C})$ with respect to the cylindrical $\sigma$-field $\mathcal{C}$ generated by the sets
\begin{align*}
\{u\in\mathcal{D}'(\bR^d)|\,(\langle u,\varphi_1\rangle,\dotso,\langle u,\varphi_N\rangle)\in B\}
\end{align*}
with $N\in \bN$, $\varphi_1,\dotso, \varphi_N\in \mathcal{D}(\bR^d)$ and $B\in\mathcal{B}(\bR^N)$. From now on we will always work with such a version.\\
The probability law of a generalized random process $s$ is given by
\begin{align*}
\mathcal{P}_s(B):=\mathcal{P}(s\in B)
\end{align*}
for $B\in\mathcal{C}$. The characteristic functional $\widehat{\mathcal{P}}_s$ is then defined by
\begin{align*}
\widehat{\mathcal{P}}_s(\varphi)=\int\limits_{\mathcal{D}'(\bR^d)}\exp(i\langle u,\varphi\rangle)d\mathcal{P}_s(u), \, \varphi\in \mathcal{D}(\bR^d).
\end{align*}
We will work with L\'{e}vy white noise, which is a generalized random process where the characteristic functional satisfies a L\'{e}vy-Khintchine representation.
\begin{definition}
A L\'{e}vy white noise $\dot{L}$ is a generalized random process, where the characteristic functional is given by
\begin{align*}
\widehat{\mathcal{P}}_{\dot L}(\varphi)=\exp\left(\,\,\int\limits_{\bR^d} \psi(\varphi(x))\lambda^d(dx)\right)
\end{align*}
for every $\varphi \in \mathcal{D}(\bR^d)$, where $\psi:\bR\to \bC$ is given by
\begin{align*}
\psi(z)=i\gamma z-\frac{1}{2}az^2+\int\limits_{\bR}(e^{ixz}-1-ixz\one_{|x|\le1})\nu(dx)
\end{align*}
where $a\in\bR^+$, $\gamma\in\bR$ and $\nu$ is a L\'{e}vy-measure, i.e. a measure such that $\nu(\{0\})=0$ and $$\int\limits_{\bR} \min(1,x^2)\nu(dx)<\infty.$$ We say that $\dot{L}$ has the characteristic triplet $(a,\gamma,\nu)$.
\end{definition}
The existence of the L\'{e}vy-white noise was proven in [\ref{Gelfand}]. The domain of the L\'{e}vy white noise can also be extended to indicator functions $\one_{A}$ for $A$ be a Borel set with finite Lebesgue measure by using the construction in [\ref{Fageot}, Proposition 3.4]. For a more general function $f$ we say that $f$ is in the domain $\dot{L}$ if there exists a sequence of elementary functions $f_n$ converging almost everywhere to $f$ such that $\langle\dot{L},f_n\one_{A}\rangle$ convergens in probability for $n\to\infty$ for every Borel set $A$ and set $\langle \dot{L},f\rangle$ as the limit in probability of $\langle\dot{L},f_n\rangle$ for $n\to\infty$, where for a elementary function $f:=\sum_{j=1}^m a_j \one_{A_j}$ $\langle\dot{L},f\rangle$ is defined by $\sum_{j=1}^ma_j\langle\dot{L}, \one_{A_j}\rangle$, see also [\ref{Fageot}, Definition 3.6]. For the maximal domain of the L\'{e}vy white noise $\dot{L}$ we write $L(\dot{L})$. By setting $L(A):=\langle \dot{L}, \one_A\rangle$ for bounded Borel sets $A$, the extention of a L\'{e}vy white noise $\dot{L}$ can be identified with a L\'{e}vy basis $L$ in the sense of Rajput and Rosinski [\ref{Rajput}], see [\ref{Fageot}, Theorem 3.5 and Theorem 3.7]. As a L\'{e}vy basis can be identified with a L\'{e}vy white noise in a canonical way, i.e. $\langle \dot{L},\varphi\rangle:=\int\limits_{\bR^d} \varphi(x)dL(x)$ for $\varphi \in\mathcal{D}(\bR^d)$, we do not differ between a L\'{e}vy basis and L\'{e}vy-white noise. In particular, a Borel-measurable function $f:\bR^d\to \bR$ is in $L(\dot{L})$ if and only if $f$ is integrable with respect to the L\'{e}vy basis $L$ in the sense of Rajput and Rosinski [\ref{Rajput}], see [\ref{Fageot}, Def. 3.6].\\
The L\'{e}vy white noise is stationary in the following sense.
\begin{definition}
A generalized process $s$ is called stationary if for every $t\in\bR^d$, $s(\cdot+t)$ has the same law as $s$. Here, $s(\cdot+t)$ is defined by
$$\langle s(\cdot+t),\varphi\rangle:=\langle s, \varphi(\cdot-t)\rangle \textrm{ for every }\varphi\in\mathcal{D}(\bR^d).$$
\end{definition}

\subsection{Generalized stochastic processes constructed from L\'{e}vy white noise}
We now state and prove our first theorem which asserts that a large class of SPDEs has a generalized solution by only assuming low moment conditions on the L\'{e}vy white noise. 
\begin{theorem}\label{theorem1} 
Let $\dot{L}$ be a L\'{e}vy white noise with characteristic triplet $(a,\gamma,\nu)$ and $G:\bR^d\to\bR$ be a measurable function such that $G\in L^1(\bR^d)$. Define 
\begin{align}\label{eq201}G_R(x):=\int\limits_{B_R(x)}|G(y)|\lambda^d(dy)\end {align} for every $x\in\bR^d$ and $R>0$ and 
\begin{align}\label{wiesonur}
h_R(x)=x \int\limits_0^{1/x} d_{G_R}(\alpha)\lambda^1(d\alpha)\textrm{ for }x>0.
\end{align} 
Assume that 
\begin{align}\label{ass2}
\int\limits_{\bR} \one_{|r|>1} h_R(|r|)\nu(dr)<\infty
\end{align} 
for every $R>0$. Then 
\begin{align}\label{farid}
s(\varphi):=\langle\dot{L},G\ast \varphi\rangle, \quad \varphi\in \mathcal{D}(\bR^d)
\end{align}
 defines a stationary generalized random process.
\end{theorem}
Observe that although $\varphi \in\mathcal{D}(\bR^d)$, $G\ast\varphi$ is in general not in $\mathcal{D}(\bR^d)$ unless $G$ has compact support. The point is that nevertheless, $s$ defined by $(\ref{farid})$ gives a generalized random process. Sufficient conditions for $(\ref{ass2})$ to hold will be treated in Example \ref{ichhabekeineahnung}.
\begin{proof}
We need to show that $G\ast \varphi \in L(\dot{L})$ and $\langle \dot{L}, G\ast\varphi_n\rangle\to \langle \dot{L}, G\ast\varphi\rangle$ as $n\to\infty$ in $L^0(\Omega)$ for a sequence $(\varphi_n)_{n\in \bN}$ converging to $\varphi$ in $\mathcal{D}(\bR^d)$. As $\langle \dot{L},G\ast \cdot \rangle$ is linear, this is equivalent to check that $\langle \dot{L}, G\ast(\varphi_n-\varphi)\rangle\to 0$ as $n\to\infty$ in $L^0(\Omega)$, which is implied by
\begin{align}
&\label{ha}\int\limits_{\bR^d} \left|\gamma \varphi_n\ast G(x)+ \int\limits_{\bR} r (\varphi_n\ast G)(x)(\one_{|r(\varphi_n\ast G)(x)|\le 1}-\one_{|r|\le 1})\nu(dr)\right|\lambda^d(dx)\to 0,\\
&\label{haha}\int\limits_{\bR^d}  \int\limits_{\bR} 1\wedge (r (\varphi_n\ast G)(x))^2\nu(dr)\lambda^d(dx)\to 0\textrm{ and }\\
 \label{hahaha}a^2&\int\limits_{\bR^d}|G\ast \varphi_n(x)|^2\lambda^d(dx)\to 0
\end{align}
for $n\to\infty$ if $\varphi_n\to 0$ for $n\to \infty$ in $\mathcal{D}(\bR^d)$, see [\ref{Rajput}, Theorem 2.7] (that $G\ast\varphi \in L(\dot{L})$ follows if the above quantities are finite).\\
Since $G\in L^1(\bR^d)$ it is easily seen that $$\int\limits_{\bR^d} \left|\gamma \varphi_n\ast G(x)\right|\lambda^d(dx)\le |\gamma|\, ||\varphi_n||_{L^1}||G||_{L^1}\to 0$$for $n\to \infty$. The other term in (\ref{ha}) will be splitted by
\begin{align*}
&\int\limits_{\bR^d} \int\limits_{\bR}|r(\varphi_n \ast G)(x)|\cdot|\one_{|r(\varphi_n\ast G)(x)|\le 1}-\one_{|r|\le 1}|\nu(dr)\lambda^d(dx)\\
=&\int\limits_{\bR^d} \int\limits_{\bR}|r(\varphi_n \ast G)(x)|\one_{|r(\varphi_n\ast G)(x)|\le 1,|r|>1}\nu(dr)\lambda^d(dx)+\int\limits_{\bR^d} \int\limits_{\bR}|r(\varphi_n \ast G)(x)|\one_{|r(\varphi_n\ast G)(x)|> 1,|r|\le1}\nu(dr)\lambda^d(dx)\\
=& \int\limits_{\bR}|r|\one_{|r|>1} \int\limits_{\bR^d} |(\varphi_n\ast G)(x)|\one_{|(\varphi_n\ast G)(x)|\le \frac{1}{|r|}}\lambda^d(dx)\nu(dr)\\
+& \int\limits_{\bR}|r|\one_{|r|\le 1} \int\limits_{\bR^d} |(\varphi_n\ast G)(x)|\one_{|(\varphi_n\ast G)(x)|> \frac{1}{|r|}}\lambda^d(dx)\nu(dr).
\end{align*}
Let us give a pointwise upper bound for the convolution. Let $R>0$ be such that $\supp(\varphi_n)\subset B_r(0)$ for some $r<R$. We then see that for every $x\in\bR^d$
\begin{align*}
(\varphi_n\ast G)(x)=\int\limits_{\bR^d} G(y)\varphi_n(x-y) \lambda^d(dy)= \int\limits_{B_R(x)} G(y)\varphi_n(x-y)\lambda^d(dy)\le G_R(x) ||\varphi_n||_{\infty}.
\end{align*}
We then obtain
\begin{align}
\nonumber d_{\varphi_n\ast G}(\alpha)=&\lambda^d\left( \{x\in\bR^d: |\varphi_n\ast G(x)|> \alpha \}\right)\\
\le& \lambda^d\left(\{ x\in\bR^d: |G_R(x)|> \alpha/||\varphi_n||_{\infty\}} \right)=d_{G_R}(\alpha/||\varphi_n||_{\infty}).\label{wichtig}
\end{align}
So we see by [\ref{Grafakos}, Exercise 1.1.10, p. 14] that
\begin{align*}
\int\limits_{\bR^d} |(\varphi_n\ast G)(x)|\one_{|(\varphi_n\ast G)(x)|
\le \frac{1}{|r|}}\lambda^d(dx)
\le&\int\limits_{0}^{\frac{1}{|r|}} d_{\varphi_n\ast G}(\alpha)\lambda^1(d\alpha)\le& \int\limits_{0}^{\frac{1}{|r|}}d_{G_R}(\alpha/||\varphi_n||_{\infty})\lambda^1(d\alpha).
\end{align*}
We see that the right hand side converges to $0$ for $n\to\infty$ and for $n$ large enough we have
\begin{align*}
\int\limits_{0}^{\frac{1}{|r|}}d_{G_R}\left(\frac{\alpha}{||\varphi_n||_{\infty}}\right)\lambda^1(d\alpha)\le \int\limits_{0}^{\frac{1}{|r|}}d_{G_R}\left(\alpha\right)\lambda^1(d\alpha)=\frac{1}{|r|}h_R(|r|).
\end{align*}
Lebesgue's dominated convergence theorem  using (\ref{ass2}) implies
\begin{align*}
 &\int\limits_{\bR}|r|\one_{|r|>1} \int\limits_{\bR^d} |(\varphi_n\ast G)(x)|\one_{|(\varphi_n\ast G)(x)|\le \frac{1}{|r|}}\lambda^d(dx)\nu(dr)\to 0
\end{align*}
for $n\to\infty$. \\
For the other term we see from Young's inequality that
\begin{align*}
\int\limits_{\bR^d} |(\varphi_n\ast G)(x)|\one_{|(\varphi_n\ast G)(x)|> \frac{1}{|r|}}\lambda^d(dx)\le |r|\cdot||\varphi_n\ast G||^2_{L^2(\bR^d)}\le |r|  \|G\|_{L^1(\bR^d)}^2\|\varphi_n\|_{L^2(\bR^d)}^2
\end{align*}
and again from Lebesgue's dominated convergence theorem (since $\int_{|r|\le 1} r^2\nu(dr)<\infty$) 
\begin{align*}
\int\limits_{\bR}|r|\one_{|r|\le 1} \int\limits_{\bR^d} |(\varphi_n\ast G)(x)|\one_{|(\varphi_n\ast G)(x)|> \frac{1}{|r|}}\lambda^d(dx)\nu(dr)\to 0
\end{align*}
for $n\to\infty$. This gives (\ref{ha}).\\
Now we check (\ref{haha}). We first note that
\begin{align*}
1\wedge (r^2 (\varphi_n\ast G)(x)^2)\le& \one_{|r(\varphi_n\ast G)(x)|>1}\one_{|r|>1}+|\varphi_n\ast G(x)| |r|\one_{|r(\varphi_n\ast G)(x)|>1}\one_{|r|\le 1}\\
&+(\varphi_n\ast G(x)r)^2 \one_{|r(\varphi_n\ast G)(x)|\le1}\one_{|r|\le 1}+|\varphi_n\ast G(x)| |r|\one_{|r(\varphi_n\ast G)(x)|\le1}\one_{|r|> 1}.
\end{align*}
From the calculations that led to (\ref{ha}) we conclude that the second and fourth term (when integrated with respect to $\nu(dr)\lambda^d(dx)$) converge to $0$ for $n\to\infty$ and for the first term we note that
\begin{align*}
\int\limits_{\bR^d}\one_{|r(\varphi_n\ast G)(x)|>1}\lambda^d(dx)&=d_{\varphi_n\ast G}\left(\frac{1}{|r|}\right)\le d_{G_R}\left(\frac{1}{|r|||\varphi_n||_{\infty}}\right)
\end{align*}
and by Lebesgue's dominated convergence theorem we conclude that
\begin{align*}
\int\limits_{\bR}\one_{|r|>1} d_{G_R}\left(\frac{1}{|r|||\varphi_n||_{\infty}}\right)\nu(dr)\to 0
\end{align*}
for $n\to\infty$, as $h_R(|r|)\ge d_{G_R}(1/|r|)$.
For the third term we easily see that
\begin{align*}
 \int\limits_{\bR}\int\limits_{\bR^d}(\varphi_n\ast G(x)r)^2 \one_{|r(\varphi_n\ast G)(x)|\le1}\one_{|r|\le 1}\lambda^d(dx)\nu(dr) \le ||\varphi_n\ast G(x)||_{L^2}^2\left( \,\,\int\limits_{\bR}\one_{|r|\le 1} |r|^2\nu(dr) \right)\to 0
\end{align*}
for $n\to\infty$. This gives (\ref{haha}). Finally, (\ref{hahaha}) follows from Young's inequality since
\begin{align*}
\|G\ast \varphi_n\|_{L^2(\bR^d)}^2\le \|G\|_{L^1(\bR^d)}^2\|\varphi_n\|_{L^2(\bR^d)}^2\to 0\,\textrm{ for }n\to\infty.
\end{align*}
The stationarity of the L\'{e}vy white noise implies the stationarity of the generalized process $s$,  as
\begin{align*}
\langle s(\cdot+ t),\varphi\rangle=\langle s, \varphi(\cdot -t)\rangle =\langle \dot{L}, G\ast \varphi(\cdot +t)\rangle=\langle \dot{L}(\cdot+(-t)),G\ast \varphi\rangle.
\end{align*}
\end{proof}
The kernel function $G$ has not always such nice integrability properties as assumed in Theorem \ref{theorem1}. For example, the Green function of the Laplacian is neither integrable nor square integrable. As this is the case, we will prove another theorem, which will assure the existence of the generalized process $s$ under some other conditions.
\begin{theorem}\label{theorem3.8}
If $G\in L^1_{loc}(\bR^d)$ such that $\|G\ast\varphi_n\|_{L^2(\bR^d)}\to 0$ for $n\to\infty$ for every sequence $(\varphi_n)_{n\in\bN}\subset \mathcal{D}(\bR^d)$ converging to $0$ and the L\'{e}vy white noise $\dot{L}$ has characterstic triplet $(a,\gamma,\nu)$ such that the first moment of $\dot{L}$ vanishes, i.e.  $\bE |\langle \dot{L}, \varphi\rangle|<\infty$ and $\bE \langle \dot{L}, \varphi\rangle=0$ for every $\varphi\in\mathcal{D}(\bR^d)$, then $s:\mathcal{D}(\bR^d)\to L^0(\Omega)$ defined by
\begin{align*}
s(\varphi):=\langle\dot{L},G\ast \varphi\rangle
\end{align*}
defines a stationary generalized process if 
\begin{align}
\label{wiesonur1}\int\limits_{\bR} \one_{|r|>1} |r|\int\limits_{\frac{1}{|r|}}^\infty d_{G_R}(\alpha)\lambda^1(d\alpha)\nu(dr)&<\infty\textrm{ and }\\
\label{wiesonur2}\int\limits_{\bR} \one_{|r|>1} |r|^2 \int\limits_{0}^{\frac{1}{|r|}} \alpha d_{G_R} (\alpha)\lambda^1(d\alpha)\nu(dr)&<\infty.
\end{align}
for all $R>0$, where $G_R$ is defined by $(\ref{eq201})$.
\end{theorem}
Observe that (\ref{ass2}) can be written as $\int_{|r|>1} |r|\int_0^{1/|r|} d_{G_R}(\alpha)\lambda^d(d\alpha)\nu(dr)$, which is slightly stronger than (\ref{wiesonur2}). However, for Theorem \ref{theorem3.8} we additionally need (\ref{wiesonur1}) and $\bE \langle \dot{L},\varphi\rangle=0$ for every $\varphi\in\mathcal{D}(\bR^d)$.
\begin{proof}
By [\ref{Sato}, Example 25.12, p. 163] we conclude that we need to show similar to Theorem \ref{theorem1} that (\ref{haha}), (\ref{hahaha}) and
\begin{align}
\label{eq101}&\int\limits_{\bR^d} \left|\int\limits_{\bR} r (\varphi_n\ast G)(x)\one_{|r(\varphi_n\ast G)(x)|> 1}\nu(dr)\right|\lambda^d(dx)\to 0,
\end{align}
are satisfied for all $(\varphi_n)_{n\in\bN}$ converging to $0$ in $\mathcal{D}(\bR^d)$.
Let $(\varphi_n)_{n\in\bN}\subset \mathcal{D}(\bR^d)$ converging to $0$ such that $\supp \varphi_n\subset B_R(0)$ for some $R>0$ and all $n\in\bN$. Using that $\int_{\bR^d} |f(x)|\one_{|f(x)|>\beta} \lambda^d(dx)=\int_{\beta}^\infty d_f(\alpha)\lambda^1(d\alpha)+\beta d_f(\beta)$ for $\beta>0$ and measurable $f$ (cf. [\ref{Grafakos}, Exercise 1.1.10, p. 14]), we estimate (\ref{eq101}) by 
\begin{align*}
&\int\limits_{\bR^d} \left|\int\limits_{\bR} r (\varphi_n\ast G)(x)\one_{|r(\varphi_n\ast G)(x)|> 1}\nu(dr)\right|\lambda^d(dx)\\
\le&\left(\int\limits_{\bR}  \one_{|r|\le 1}|r|^2\nu(dr)\right)\|G\ast\varphi_n\|_{L^2(\bR^d)}^2+\int\limits_{\bR}\one_{|r|>1}|r|\int\limits_{\bR^d}  |(\varphi_n\ast G)(x)|\one_{|r(\varphi_n\ast G)(x)|> 1}\lambda^d(dx)\nu(dr)\\
=&\left(\int\limits_{\bR}  \one_{|r|\le 1}|r|^2\nu(dr)\right)\|G\ast\varphi_n\|_{L^2(\bR^d)}^2+\int\limits_{\bR}\one_{|r|>1}|r|\int\limits_{\frac{1}{|r|}}^{\infty} d_{\varphi_n\ast G}(\alpha)\lambda^1(d\alpha)\nu(dr)+\int\limits_{\bR}\one_{|r|>1}d_{\varphi_n\ast G}(1/|r|)\nu(dr)\\
&\to 0\\
\end{align*} 
for $n\to\infty$ by Lebesgue's dominated convergence, where we used that by (\ref{wichtig})
\begin{align*}
&\int\limits_{\bR}\one_{|r|>1}|r|\int\limits_{\frac{1}{|r|}}^{\infty} d_{\varphi_n\ast G}(\alpha)\lambda^1(d\alpha)\nu(dr)+\int\limits_{\bR}\one_{|r|>1}d_{\varphi_n\ast G}(1/|r|)\nu(dr)\\
\le&\int\limits_{\bR}\one_{|r|>1}|r|\int\limits_{\frac{1}{|r|}}^{\infty} d_{G_R}(\alpha/\|\varphi_n\|_{\infty})\lambda^1(d\alpha)\nu(dr)+\int\limits_{\bR}\one_{|r|>1}d_{G_R}(1/(|r|\|\varphi_n\|_{\infty}))\nu(dr)\\
\le&\int\limits_{\bR}\one_{|r|>1}|r|\int\limits_{\frac{1}{|r|}}^{\infty} d_{G_R}(\alpha)\lambda^1(d\alpha)\nu(dr)+\int\limits_{\bR}\one_{|r|>1}d_{G_R}(1/|r|)\nu(dr)
\end{align*}
for large $n$ and the latter is finite by (\ref{wiesonur1}), (\ref{wiesonur2}) and
\begin{align*}
\int\limits_{0}^{x} \alpha d_{G_R}(\alpha)\lambda^1(d\alpha)\ge d_{G_R}(x)\int\limits_{0}^x \alpha\lambda^1(d\alpha)=\frac{1}{2}d_{G_R}(x) x^2\textrm{ for every }x>0.
\end{align*}
This gives (\ref{eq101}). We control $(\ref{haha})$ by
\begin{align*}
&\int\limits_{\bR^d}  \int\limits_{\bR} 1\wedge (r (\varphi_n\ast G)(x))^2\nu(dr)\lambda^d(dx)\\
\le&\int\limits_{\bR^d}  \int\limits_{\bR}\one_{|r(\varphi_n\ast G)(x)|>1}\one_{|r|>1}+|\varphi_n\ast G(x)|^2 |r|^2\one_{|r|\le 1}+|\varphi_n\ast G(x)|^2 |r|^2\one_{|r(\varphi_n\ast G)(x)|\le1}\one_{|r|> 1}\nu(dr)\lambda^d(dx)\\
=:&I_1+I_2+I_3.
\end{align*}
We have already shown how to control $I_1$ and $I_2$, so we only need to show that $I_3$ converges to $0$ for $n\to\infty$. We see by [\ref{Grafakos}, Exercise 1.1.10] that
\begin{align*}
&\int\limits_{\bR^d}  \int\limits_{\bR}|\varphi_n\ast G(x)|^2 |r|^2\one_{|r(\varphi_n\ast G)(x)|\le1}\one_{|r|> 1}\nu(dr)\lambda^d(dx)\\
\le&2\int\limits_{\bR} \one_{|r|>1} r^2\int\limits_{0}^{\frac{1}{|r|}} \alpha d_{\varphi_n\ast G}(\alpha) \lambda^1(d\alpha)  \nu(dr)\to 0\\
\end{align*}
by using that
\begin{align*}&\int\limits_{\bR} \one_{|r|>1} r^2\int\limits_{0}^{\frac{1}{|r|}} \alpha d_{\varphi_n\ast G}(\alpha) \lambda^1(d\alpha)  \nu(dr)\le\int\limits_{\bR} \one_{|r|>1}r^2 \int\limits_{0}^{\frac{1}{|r|}} \alpha d_{G_R}(\alpha) \lambda^1(d\alpha) \nu(dr)<\infty\\
\end{align*}
for large $n$ by (\ref{wichtig}) and by assumption. Hence, we conclude that $s$ defines a generalized process. Stationarity follows by the same arguments as in the proof of Theorem \ref{theorem1}.
\end{proof}
\begin{remark}
If for every $R>0$ there exists a bounded Borel set $A_R$  and a constant $C_R>0$ such that $G_R(x)\le C_R G(x)$ for all $x\in\bR^d \setminus A_R$, then we can replace $G_R$ by $G$ in (\ref{wiesonur}), (\ref{wiesonur1}) and (\ref{wiesonur2}).This follows from the estimate $d_{G_R}(\alpha)\leq \lambda^d(A_R)+d_G(\alpha/C_R)$ for (\ref{wiesonur}) and (\ref{wiesonur2}), and for (\ref{wiesonur1}) one can argue similarly to the proof of Example \ref{ichhabekeineahnung} below, using the boundedness of $G_R$ on a set $A_{2R}$ related to $A$.
\end{remark}
\begin{remark}
Under certain conditions one can replace $h_R(|r|)$ in (\ref{ass2}) by $d_{G_R}(1/|r|)$, for example if for every $R>0$, $d_{G_R} \in L^p([0,1])$ for some $p>1$ and $d_{G_R}(x)x^{1/p}\ge C$ for some constant $C>0$ independent of $x$. This follows by
\begin{align*}
\frac{1}{xd_{G_R}(x)}\int\limits_{0}^x d_{G_R}(\alpha ) \lambda^1(d\alpha)\le \frac{x^{1-1/p}}{xd_{G_R}(x)}\| d_{G_R}\|_{L^p([0,1])}\le \frac{1}{C}\| d_{G_R}\|_{L^p([0,1])}<\infty \textrm{ for all }x\in(0,1). 
\end{align*}
\begin{example}\label{ichhabekeineahnung}
We will discuss now two examples. For the first example, we assume that $G  \in L^1_{loc}(\bR^d)$  and there exist $\beta>d/2$, $C>0$ and a bounded, open set $A$ with $0\in A$ such that $|G(x)|\le C \|x\|^{-\beta}$ for all $x\in \bR^d\setminus A$. We find that
\begin{align*}
d_{G_R}(\alpha)\le C' (\alpha^{-\frac{d}{\beta}}+\one_{\alpha\le \|G_R\|_{L^\infty(A_{2R})}}),
\end{align*}
where $A_{2R}:=\{x\in \bR^d: dist(x,A)\le 2R\}$. We conclude 
\begin{align*}
\int\limits_{\frac{1}{|r|}}^\infty d_{G_R}(\alpha)\lambda^1(d\alpha)\le \tilde{C}|r|^{\frac{d}{\beta}-1} +C'\max\{\|G_R\|_{L^\infty(A_{2R})}-\frac{1}{|r|},0\}
\end{align*}
and
\begin{align*}
\int\limits_{0}^{\frac{1}{|r|}} \alpha d_{G_R}(\alpha) \lambda^1(d\alpha)\le \tilde{C} \left(|r|^{\frac{d}{\beta}-2}+|r|^{-2}\right)
\end{align*}
for some constant $\tilde{C}>0$ for all $|r|>1$. Writing $G=G\one_{B_M(0)}+G\one_{\bR^d\setminus B_M(0)}$ for large $M$, we have $G\one_{B_M(0)}\in L^1(\bR^d)$ and $G\one_{\bR^d\setminus B_M(0)}\in L^2(\bR^d)$ and since $\varphi_n\in L^1(\bR^d)\cap L^2(\bR^d)$ we obtain from Young's inequality that $\|G\ast\varphi_n\|_{L^2(\bR^d)}\to 0$, $n\to\infty$. If $\int\limits_{|r|>1} |r|^{\frac{d}{\beta}}\nu(dr)< \infty$,  we conclude by Theorem \ref{theorem3.8} (if $\dot{L}$ satisfies the assumptions specified there) that $s(\varphi):=\langle \dot{L}, G\ast \varphi\rangle $ defines a generalized random process.\\
For the second example we assume that
\begin{align*}
e^{c||x||}G(x)\in L^2(\bR^d)
\end{align*}
for some constant $c>0$. By the H\"older inequality we conclude
\begin{align*}
\int\limits_{\bR^d} |G(x)|\lambda^d(dx)\le ||\exp(-c||\cdot||)||_{L^2}\cdot|| \exp(c||\cdot||)G(\cdot)||_{L^2}<\infty
\end{align*}
and
\begin{align*}
\int\limits_{B_R(x)}|G(y)|\lambda^d(dy)\le ||e^{c||\cdot||}G||_{L_2}\left(\int\limits_{B_R(x)}e^{-2 c|| y||}\lambda^d(dy)\right)^{1/2}\le C_R \exp(-c||x||)
\end{align*}
for some constant $C_R>0$.  Hence,
\begin{align*}
d_{G_R}(\alpha)\le d_{\exp(-c||\cdot||)}\left(\frac{\alpha}{C_R}\right),
\end{align*}
for $\alpha>0$. We conclude that for $r\ge C_R$, 
\begin{align*}
\int\limits_{0}^{1/|r|} d_{G_R}(a)\lambda^1(da)&\le\int\limits_{0}^{\frac{1}{|r|}} C_d\left(\frac{\log\left(\frac{C_R}{\alpha}\right)}{c}\right)^d \lambda^1(d\alpha)\\&=\frac{C_d}{c^d} C_R\Gamma (d+1,\log(C_R|r|)) \\
&=\frac{C}{|r|}\sum\limits_{k=0}^d \frac{\log(C_R|r|)^k}{k!}
\end{align*}
for some finite constants $C_d$ and $C$, where $\Gamma(d+1,z)=\int_z^\infty t^d e^{-t}\lambda^1(dt)$ denotes the upper incomplete gamma function. Assuming $\int\limits_{|r|>1}\log(|r|)^d\nu(dr)<\infty$, we conclude
\begin{align*}
\int\limits_{\bR}\one_{|r|>1/C_R}  \left(C\sum\limits_{k=0}^d \frac{\log(C_R|r|)^k}{k!}\right)\nu(dr)<\infty
\end{align*}
and by Theorem \ref{theorem1} we obtain that $s$ defined as above defines a generalized process.
\end{example}
\end{remark}
Until now we have only given sufficient conditions for the existence of a generalized process $s$ defined by a convolution with a suitable kernel $G$. We will give a necessary condition if $G$ is positive in $\bR^d$.
\begin{corollary}\label{cor3.9}
Let $G\in L^1_{loc}(\bR^d)$ such that $G(x)\ge 0$ $\lambda^d-$a.e (or $G(x)\le 0$ $\lambda^d-$a.e.). Let $\dot{L}$ be a L\'{e}vy white noise with characteristic triplet $(a,\gamma,\nu)$. If $s:\mathcal{D}(\bR^d)\to L^0(\Omega)$ defined by $s(\varphi):=\langle \dot{L},G\ast\varphi\rangle$ for $\varphi\in\mathcal{D}(\bR^d)$ defines a generalized process, then 
\begin{align*}
\int\limits_{\bR} \one_{|r|>1} d_{G_R}\left( \frac{1}{|r|} \right)\nu(dr)<\infty
\end{align*} 
for every $R>0$ and $G_R$ defined by $(\ref{eq201})$.
\end{corollary}
\begin{proof}
We know that for $\varphi\in\mathcal{D}(\bR^d)$, it is necessary for $G\ast \varphi \in L(\dot{L})$ that (cf. [\ref{Rajput}, Theorem 2.7, p.461-462])
\begin{align}\label{whynot}
\nonumber\infty>\int\limits_{\bR}\int\limits_{\bR^d}(1\wedge (r\varphi\ast G)(x)^2)\lambda^d(dx)\nu(dr)\ge& \int\limits_{\bR}\int\limits_{\bR^d} \one_{|r|>1}\one_{|r\varphi\ast G|>1}\lambda^d(dx)\nu(dr)\\
=& \int\limits_{\bR}\one_{|r|>1} d_{G\ast\varphi}\left(1/|r|\right)\nu(dr).
\end{align}
Now let $\varphi\in\mathcal{D}(\bR^d)$, $\varphi\ge 0$ such that $\varphi\ge 1$ in $B_R(0)$. We see that
\begin{align*}
  d_{G\ast\varphi}(\alpha)&=\lambda^d\left(\left\{x\in\bR^d: \int\limits_{\bR^d} G(x-y)\varphi(y)\lambda^d(dy)>\alpha\right\}\right)\\
&\ge\lambda^d\left(\left\{x\in\bR^d: \int\limits_{B_R(x)} G(y)\lambda^d(dy)>\alpha\right\}\right)=d_{G_R}(\alpha).
\end{align*}
By assumption we conclude
\begin{align*}
\int\limits_{\bR}\one_{|r|>1} d_{G_R}\left(1/|r|\right)\nu(dr)\le \int\limits_{\bR}\int\limits_{\bR^d}(1\wedge (r\varphi\ast G)(x)^2)\lambda^d(dx)\nu(dr)<\infty.
\end{align*}
\end{proof}
\subsection{Homogeneous Elliptic SPDEs} Let $p(z)=\sum_{|\alpha|\leq m} a_{\alpha} z^{\alpha}$ be a polynomial in $d$ variables and $\dot{L}$ some L\'{e}vy noise. We are interested in generalized solutions of the stochastic partial differential equation
\begin{align}\label{elliptic}
p(D)s=\dot{L}.
\end{align}
This means formally
$$ \langle p(D)s,\varphi\rangle=\langle \dot{L} ,\varphi\rangle \quad \forall\,\varphi \in\mathcal{D}(\bR^d).$$
We interprete the left-hand side of this equation as $\langle s,p^*(D)\varphi\rangle$, where $p^*(D)$ denotes the formal adjoint operator of $p(D)$, which is known to be given by $p(-D)$. Hence, by definition, by a solution of (\ref{elliptic}) we mean a generalized process $s$ that satisfies
$$\langle s, p^*(D)\varphi\rangle=\langle \dot{L},\varphi\rangle\quad\forall \;\varphi\in\mathcal{D}(\bR^d).$$
Let $G$ be a fundamental solution of the operator $p^*(D)$, i.e. a distribution such that $p^*(D) G\ast \varphi=\varphi$ for every $\varphi\in\mathcal{D}(\bR^d)$. By the theorem of Malgrange-Ehrenpreis, such a fundamental solution always exists. Suppose this fundamental solution arises actually from a locally integrable function $G$ such that the assumptions of Theorem \ref{theorem1} or Theorem \ref{theorem3.8} are satisfied. If we then define the generalized process $s$ by $\langle s,\varphi\rangle:=\langle \dot{L},G\ast \varphi\rangle$ for $\varphi\in\mathcal{D}(\bR^d)$, then this defines a generalized process that satisfies (\ref{elliptic}). This follows from the simple calculation
\begin{align*}
\langle p(D)s,\varphi\rangle=\langle s,p^*(D)\varphi\rangle=\langle \dot{L},G\ast (p^*(D)\varphi)\rangle =\langle \dot{L},p^*(D)G\ast\varphi\rangle=\langle \dot{L},\varphi\rangle.
\end{align*}
To find conditions when Theorem \ref{theorem3.8} can be applied, we specialise to homogeneous elliptic partial differential operators. We say that a polynomial is elliptic homogeneous of degree $m$ if $p(z)=\sum\limits_{|\alpha|= m} a_{\alpha} z^{\alpha}$ and $p(z)\neq 0$ for all $z\in\bR^d\setminus \{0\}$. We call $p(D)$ then an elliptic homogenous partial differential operator of degree $m$. Observe that in this case the adjoint operator is given by $p^*(D)=(-1)^mp(D)$. Hence, the fundamental solution of $p^*(D)$ and $p(D)$ differ only by the factor $(-1)^m$. We now have:
\begin{proposition} Let $p(D)$ be an elliptic homogeneous partial differential operator of order $m\in\bN$. If $d>2m$ and the L\'{e}vy white noise $\dot{L}$ with characteristic triplet $(a,\gamma,\nu)$ satisfies
\begin{align*}
\int\limits_{\bR} \one_{|r|>1} |r|^{\frac{d}{d-m}+\varepsilon}\nu(dr)<\infty
\end{align*}
for some $\varepsilon>0$ and the first moment of $\dot{L}$ vanishes, then there exists a generalized process $s$ which solves the SPDE $(\ref{elliptic})$.
\end{proposition}
\begin{proof}
It is known that in that case, the fundamental solution arises from a locally integrable function $G$ that satisfies $|G(x)|\le c \|x\|^{m-d}\log(\|x\|)$ for all $\|x\|\ge 2$ and some constant $c>0$, see [\ref{Ortner}, Proposition 2.4.8, p. 155]. The rest follows by Example \ref{ichhabekeineahnung}.
\end{proof}
\begin{remark}
In the case of the Laplacian $\Delta$, when $d\ge 5$, with methods similar to the proof of Example \ref{ichhabekeineahnung} one can show that it is enough that 
\begin{align}\label{eq10101}
\int\limits_{\bR} \one_{|r|>1} |r|^{\frac{d}{d-2}}\nu(dr)<\infty
\end{align}
for the existence of a generalized solution. Moreover, if we choose for $\Delta$ the fundamental solution $G(x)=c_d |x|^{2-d}$, where $c_d\in\bR\setminus\{0\}$, then by Corollary \ref{cor3.9} it is also necessary that $(\ref{eq10101})$ holds true for $\langle \dot{L},G\ast\varphi\rangle$ to define a generalized solution.
\end{remark}
\section{CARMA generalized processes}\label{section2}
We construct a generalization of CARMA processes. A CARMA generalized process is a generalized solution of a special SPDE.
\begin{definition}\label{carma}
Let $\dot{L}$ be a L\'{e}vy white noise, $n,m\in\bN_0$ and $p,q:\bR^d\to\bR$ be polynomials of the form
\begin{align*}
p(x)=\sum\limits_{|\alpha|\le n} p_{\alpha}x^{\alpha}
\end{align*}
and
\begin{align*}
q(x)=\sum\limits_{|\alpha|\le m}q_{\alpha}x^{\alpha}.
\end{align*}
A generalized process $s: \mathcal{D}(\bR^d)\to L^0(\Omega)$ is called a CARMA$(p,q)$ generalized process if $s$ solves the equation
\begin{align*}
p(D)s=q(D)\dot{L},
\end{align*}
which means that
\begin{align}\label{spde}
\langle s,p(D)^*\varphi\rangle =\langle \dot{L},q(D)^*\varphi\rangle\textrm{ a.s. for every }\varphi\in\mathcal{D}(\bR^d).
\end{align}
\end{definition}
Recall that $p(D)^*=p(-D)$ and $q(D)^*=q(-D)$.
For classical CARMA processes in dimension 1 the assumptions on the polynomials are that $q/p$ has only removable singularities on the imaginary axis and the degree of the polynomial $p$ is higher than the degree of $q$, which implies that $\|q/p\|_{L^2(i\bR)}<\infty$. For a detailed discussion see [\ref{Lindner}]. In dimension 1 CARMA generalized processes were defined in [\ref{Hannig}], where the white noise was assumed to be Gaussian and the polynomial $p$ has no zeroes on the imaginary axis, see  [\ref{Hannig}, Proposition 2.5, p. 3616]. All the assumptions above imply even more, namely that $q/p$ has a holomorphic extension on the strip $\{z\in \bC: |\Re z|<\varepsilon\}$ for a small $\varepsilon>0$. We take this as an assumption also for higher dimensions $d$:
\begin{assumption}\label{ass1}
The rational function $q(i\cdot)/p(i\cdot)$ has a holomorphic extension in a strip $\{z\in\bC^d: \|\Im z\|<\varepsilon\}$ for some $\varepsilon>0$.
\end{assumption}
This assumption implies especially that there exist two polynomials $h$ and $l$ such that $h(i\cdot)/l(i\cdot)=p(i\cdot)/q(i\cdot)$ and $l(i\cdot)$ has no zeroes in the strip $\{z\in\bC^d: \|\Im z\|\le\varepsilon/2\}$. Hence we may and do assume for the rest of this section that $h=p$ and $l=q$.\\
We prove an existence theorem under mild moment conditions.

\begin{theorem}\label{theoremcarma}
Let $p,q$ be polynomials as in Definition $\ref{carma}$ such that the Assumption \ref{ass1} holds true. Furthermore, let $\dot{L}$ be a L\'{e}vy white noise with characteristic triplet $(a,\gamma,\nu)$ such that $$\int\limits_{\bR}\one_{|r|>1} \log(|r|)^d\nu(dr).$$ Then there exists a stationary CARMA$(p,q)$ generalized process.
\end{theorem}
\begin{proof}
Under the Assumption \ref{ass1} it follows by arguments similar as in the proof of [\ref{Hormander}, Lemma 2, p. 557] that there exists an $\alpha\in\bN$ and $\delta>0$ such that
$$
\sup_{\|\eta\|\le \delta} \left\|  \frac{q(-i\cdot+\eta)}{p(-i\cdot+\eta)\psi(\cdot+i\eta)}\right\|_{L^2(\bR^d)}<\infty,
$$
where $$\psi(z):=\left(1+\sum\limits_{j=1}^d z_i^2\right)^{\alpha}.$$
It follows by a Paley-Wiener theorem (e.g. [\ref{Reed}, Theorem XI.13, p.18]) that the inverse Fourier transform $G$ of $\frac{q(-i\cdot)}{\psi(\cdot)p(-i\cdot)}$ satisfies 
\begin{align*}
e^{c||x||}G(x)\in L^2(\bR^d)
\end{align*}
for some $0<c<\delta$. Observe that $G$ is indeed real-valued, as $ \frac{q(-i\cdot)}{p(-i\cdot)\psi(\cdot)}=\overline{ \frac{q(i\cdot)}{p(i\cdot)\psi(-\cdot)}}$.
Observe that $\psi$ is a continuous Fourier multiplier from $\mathcal{D}(\bR^d)$ to $\mathcal{D}(\bR^d)$, as $$\mathcal{F}^{-1}(\psi(\cdot)\mathcal{F}\varphi)=(1-\Delta)^{\alpha}\varphi.$$ 
By Example \ref{ichhabekeineahnung} follows that $s$ defined by
\begin{align}\label{constructed process}
\langle s,\varphi\rangle:=\left\langle \dot{L}, G\ast \mathcal{F}^{-1}\left(\psi(\cdot)\mathcal{F}\varphi\right)\right\rangle\textrm{ for every }\varphi\in\mathcal{D}(\bR^d)
\end{align}
 defines a generalized process and by similar arguments to the proof of Theorem \ref{theorem1} it follows that $s$ is stationary. Now let $\varphi\in\mathcal{D}(\bR^d)$, we conclude by $\mathcal{F}p(-D)\varphi=p(-i\cdot)\mathcal{F}\varphi$ for all $\varphi\in\mathcal{D}(\bR^d)$ that
\begin{align*}
\langle s,p(D)^*\varphi\rangle&=\left\langle \dot{L}, \left(G\ast \mathcal{F}^{-1}\left(\psi(\cdot)\mathcal{F}(p(D)^*\varphi)\right)\right)\right\rangle\\
&=\left\langle \dot{L}, \mathcal{F}^{-1} \left(\psi(\cdot)\frac{q(-i\cdot)}{\psi(\cdot)p(-i\cdot)}p(-i\cdot)\mathcal{F}\varphi\right)\right\rangle\\
&=\langle \dot{L},\mathcal{F}^{-1}\left(q(-i\cdot)\mathcal{F}\varphi\right)\rangle=\langle\dot{L},q(D)^*\varphi\rangle.
\end{align*}
\end{proof}
\begin{remark} Under the assumptions of Theorem \ref{theoremcarma} the only solutions of $(\ref{spde})$ are of the form $s+u$, where $s$ is the solution constructed in Theorem \ref{theoremcarma} and $u$ solves the equation $\langle u,p(D)^*\varphi\rangle=0$ a.s. for every $\varphi\in\mathcal{D}(\bR^d)$.
\end{remark}
We obtain directly the following corollary, which generalizes [\ref{Hannig}, Proposition 2.5, p. 3616] from Gaussian noise to L\'{e}vy white noise.
\begin{corollary}
Let $d=1$ and $p(z)=\prod_{j=1}^n (p_j-z)$ and $q(z)=\prod_{j=1}^m (q_j-z)$ be two real polynomials, such that $p/q$ has no roots on the imaginary axis. Then there exists a stationary generalized solution $s:\mathcal{D}(\bR^d)\to L^0(\Omega)$ of the equation 
$p(\frac{d}{dx})s=q(\frac{d}{dx})\dot{L}$ for every L\'{e}vy white noise $\dot{L}$ with characteristic triplet $(a,\gamma,\nu)$ such that $\int_{|r|>1} \log(|r|)\nu(dr)$. 
\end{corollary}
\begin{example}
Let us look at the polynomial $p(iz):=-\lambda-\sum\limits_{j=1}^dz_j^2$ for $d\in\bN$ with $\lambda>0$, which corresponds to the partial differential operator $L=-\lambda+\Delta$. The real part is given by $\Re\, p(iz)=-\lambda-\sum\limits_{j=1}^{d} ((\Re\, z_j)^2-(\Im \,z_j)^2)$, from which we conclude that $p(i\cdot)$ has no roots in $\{z\in\bC^d:\, \|\Im\,z\|^2<\lambda\}$ . It follows that for every polynomial $q$ there exists a generalized solution $s:\mathcal{D}(\bR^d)\to L^0(\Omega)$ of 
\begin{align}\label{poisson}
(-\lambda +\Delta)s=q(D)\dot{L}. 
\end{align}
\end{example}
\begin{example}
Let $p(D)=\prod_{j=1}^d (\lambda_j-\partial_{x_j})^{\alpha_j}$, $\alpha_j\in\bN_0$ for all $j\in\{1,\cdots,d\}$ and $|\lambda_j|>0$. Then its corresponding polynomial is given by $p(iz)=\prod_{j=1}^d (\lambda_j-i z_j)^{\alpha_j}$ and by Theorem \ref{theoremcarma} we find a generalized solution of the equation $p(D)s=q(D)\dot{L}$ for every partial differential operator $q(D)$, as $1/p(i\cdot)$ is holomorphic in $\{z\in \bC^d: \|\Im z\|< \varepsilon\}$ for some $\varepsilon>0$.
\end{example}
\section{CARMA random fields}\label{sectionCRF}
Until now we have only studied generalized solutions of the CARMA SPDE $(\ref{eqberger})$, but in the vast literature of stochastic partial differential equations driven by L\'{e}vy noise the concept of mild solutions seems to be more used, as the mild solution is itself a random field. We show under stronger conditions  the existence of a mild solution of $(\ref{eqberger})$. But first we recall what a mild solution is.
\begin{definition}[see {[\ref{Walsh}]}]
Let $p(D)$ and $q(D)$  be partial differential operators and let $G:\bR^d\to\bR$ be a locally integrable fundamental solution of the equation $p(D)u=q(D)\delta_0$, which means that for every $\varphi\in\mathcal{D}(\bR^d)$, $p(D)G\ast\varphi=q(D)\varphi$. We say that $X=(X_t)_{t\in\bR^d}$ defined by 
\begin{align*}
X_t=\int\limits_{\bR^d}G(t-s)\,dL(s),
\end{align*}
where $dL$ denotes a L\'{e}vy basis, is the mild solution of the equation $p(D)X=q(D)dL$, provided that the integral exists. Observe that it is necessary that $G$ is a function.
\end{definition}
We know already that $\dot{L}$ can be extended to a L\'{e}vy basis, see  [\ref{Fageot}]. We state our first result, which follows directly from the proofs  of Theorem \ref{theorem1} and Corollary \ref{cor3.9}.
\begin{proposition}\label{proposition2}\mbox{}
\begin{itemize}\item[i)]
Let $G:\bR^d\to\bR$ be a measurable function with $G\in L^1(\bR^d)\cap L^2(\bR^d)$. We define $$h(x):=x\int\limits_{0}^{1/x} d_{G}(a)\lambda^1(da)\textrm{ for }x>0.$$ Let $L$ be a L\'{e}vy basis (equivalently $\dot{L}$ a L\'{e}vy white noise) with characteristic triplet $(a,\gamma,\nu)$, and assume that 
\begin{align*}
\int\limits_{\bR} \one_{|r|>1} h(|r|)\nu(dr)<\infty.
\end{align*} 
Then the integral $$X_t=\int_{\bR^d}G(t-s)dL(s)$$ exists and defines a stationary random field $(X_t)_{t\in\bR^d}$.
\item[ii)] Conversely, if $G:\bR^d\to\bR$ is measurable and the integral $\int\limits_{\bR^d}G(-s)dL(s)$ exists, then necessarily
\begin{align*}
\int\limits_{\bR} \one_{|r|>1} d_{G}(1/|r|) \nu(dr)<\infty.
\end{align*}
\end{itemize}
\end{proposition}
\begin{proof}
By [\ref{Rajput}, Theorem 2.7], the integral $\int\limits_{\bR^d}G(t-s)dL(s)$ exists if and only if
\begin{align*}
&\int\limits_{\bR^d} \left( \gamma G(x)+\int\limits_{\bR} rG(x)(\one_{|rG(x)|\le 1}-\one_{|r|\le 1})\nu(dr)\right)\lambda^d(dx)<\infty,\\
&\int\limits_{\bR^d}\int\limits_{\bR}1\wedge (rG(x))^2\nu(dr)\lambda^d(dx)<\infty\textrm{ and}\\
&\int\limits_{\bR^d} a|G(x)|^2\lambda^d(dx)<\infty.
\end{align*}
That the conditions specified in (i) are sufficient then follows by calculations similar to those in the proof of Theorem \ref{theorem1}, while necessity of the condition specified in (ii) follows as in (\ref{whynot}). That $X_t$ as defined in (i) is stationary is clear. 
\end{proof}
Now we conclude that there exists a mild solution of the CARMA$(p,q)$ SPDE under some further restrictions.
\begin{theorem}\label{cor1}
Let $dL$ be a L\'{e}vy basis in $\bR^d$ with characteristic triplet $(a,\gamma,\nu)$ such that $\int\limits_{\bR} \one_{|r|>1}\log(|r|)^d \nu(dr)<\infty$. Assume furthermore that there exists $\varepsilon>0$ such that \begin{align}\sup_{\eta\in B_{\varepsilon}(0)}\left\|\frac{q(i\cdot+\eta)}{p(i\cdot+\eta)}\right\|_{L^2}<\infty .\end{align} Then there exists a mild solution of the equation
\begin{align}\label{carma100}
p(D)X=q(D)\,dL,
\end{align}
which is given by
\begin{align}\label{mild}
X_t=\int\limits_{\bR^d} \mathcal{F}^{-1} \left(\frac{q(i\cdot)}{p(i\cdot)}\right)(t-x)dL(x),\,t\in\bR^d.
\end{align}
\end{theorem}
\begin{proof}
Taking Fourier transforms, it is easy to check that $G:=\mathcal{F}^{-1}\frac{q(i\cdot)}{p(i\cdot)}$ is a fundamental solution of $p(D)u=q(D)\delta_0$. By [\ref{Reed}, Theorem XI.13, p.18] we see that $e^{c\|\cdot\|}G\in L^2(\bR^d)$ for all $0<c<\varepsilon$ and $G$ is real-valued by the same argument as in Theorem \ref{theoremcarma}. It follows that
\begin{align*}
h(r)=r\int\limits_{0}^{1/r} d_G(\alpha)\lambda^1(d\alpha)\le d_{G\exp(c\|\cdot\|)}(1)+r\int\limits_{0}^{1/r}d_{\exp(-c\|\cdot\|)}(\alpha)\lambda^1(d\alpha)
\end{align*}
The rest follows by Proposition \ref{proposition2} and similar calculations as in Example \ref{ichhabekeineahnung}.
\end{proof}
\begin{example}
Let $d=1,2,3$,  $\lambda>0$ and $p(D)=(\lambda-\Delta)$. We see that $$\sup_{\|\eta\|\le \lambda/2}\|1/p(i\cdot+\eta)\|_{L^2(\bR^d)}<\infty$$and by Theorem \ref{cor1} we conclude that there exists a mild solution of the equation $(\lambda-\Delta)X=dL$.
\end{example}
\begin{example}
The causal CARMA random field constructed in [\ref{Pham}, Definition 3.3] and [\ref{Kluppelberg}, Definition 2.1]  is the mild solution of the equation $P(D)X=Q(D)dL$, where $P$ and $Q$ are given in [\ref{Kluppelberg}, Proposition 2.5]. We observe that $P$ and $Q$ satisfy the assumption of Theorem \ref{cor1}, so that the causal CARMA random field of [\ref{Pham},\ref{Kluppelberg}] can be seen as a special case of CARMA random fields defined in the present paper. 
\end{example}
In classical analysis, a locally integrable function $f:\bR^d\to\bR$ specifies a distribution $T_f$ by $T_f(\varphi):=\int\limits_{\bR^d}f(x)\varphi(x)\lambda^d(dx)$ for $\varphi\in\mathcal{D}(\bR^d)$. It is now natural to ask if a mild solution $X$ of $p(D)X=q(D)dL$ also gives rise to a generalized solution of $p(D)X=q(D)\dot{L}$ via $\langle X,\varphi\rangle:=\int\limits_{\bR^d} X_s\varphi(s)\lambda^d(ds)$.\\
That this indeed the case, at least under some weak conditions which allow the application of a stochastic Fubini theorem, is the contents of the next proposition.
\begin{proposition}\label{prop300}
Let $dL$ be a L\'{e}vy basis with existing first moment and $p$ and $q$ be as in Theorem \ref{cor1}. Let 
\begin{align*}
G:=\mathcal{F}^{-1}\left(\frac{q(i\cdot)}{p(i\cdot)}\right).
\end{align*}
Then the mild solution
\begin{align*}
X_s=\int\limits_{\bR^d} G(s-u)dL(u),\,s\in\bR^d,
\end{align*}
of (\ref{mild}) gives rise to a generalized solution $X$ of the SPDE $p(D)X=q(D)\dot{L}$ via
\begin{align*}
\langle X,\varphi\rangle:=\int\limits_{\bR^d} X_s\varphi(s)\lambda^d(ds),\,\varphi\in\mathcal{D}(\bR^d).
\end{align*}
\end{proposition}
\begin{proof}
Observe that $G\in L^1(\bR^d)\cap L^2(\bR^d)$ by the proof of Theorem \ref{cor1}. We see that for every $\varphi\in\mathcal{D}(\bR^d)$
\begin{align*}
&|r\varphi(t)G(t-s)|\wedge |r\varphi(t)G(t-s)|^2\\
=&\one_{|r\varphi(t)G(t-s)|>1}|r\varphi(t)G(t-s)|+\one_{|r\varphi(t)G(t-s)|\le 1}|r\varphi(t)G(t-s)|^2\\
\le&\one_{|r|>1}\one_{|r\varphi(t)G(t-s)|>1}|r\varphi(t)G(t-s)|+\one_{|r|\le1}\one_{|r\varphi(t)G(t-s)|>1}|r\varphi(t)G(t-s)|^2\\
&+\one_{|r|\le 1}\one_{|r\varphi(t)G(t-s)|\le 1}|r\varphi(t)G(t-s)|^2+\one_{|r|> 1}\one_{|r\varphi(t)G(t-s)|\le 1}|r\varphi(t)G(t-s)|\\
= &\one_{|r|>1}|r\varphi(t)G(t-s)|+\one_{|r|\le1}|r\varphi(t)G(t-s)|^2.
\end{align*}
Since
\begin{align*}
\int\limits_{\bR^d}\int\limits_{\bR^d}\int\limits_{\bR}\one_{|r|>1}|r\varphi(t)G(t-s)|\nu(dr)\lambda^d(ds)\lambda^d(dt)&\le\int\limits_{\bR}\one_{|r|>1}|r|\nu(dr)\| |\varphi|\ast|G|\|_{L^1}<\infty\textrm{ and }\\
\int\limits_{\bR^d}\int\limits_{\bR^d}\int\limits_{\bR}\one_{|r|\le 1} |r\varphi(t)G(t-s)|^2\nu(dr)\lambda^d(ds)\lambda^d(dt)&\le\int\limits_{\bR}\one_{|r|\le 1}|r|^2\nu(dr)\| |\varphi|^2\ast|G|^2\|_{L^1}<\infty
\end{align*}
by Young's inequality and by assumption we conclude from a stochastic Fubini result ([\ref{Nielsen}, Theorem 3.1 and Remark 3.2, p. 926]; observe that $\varphi$ has compact support and that $\lambda^d$ is finite on the support of $\varphi$) that
\begin{align*}
\langle X,\varphi\rangle:&=\int\limits_{\bR^d} X_s\varphi(s)\lambda^d(ds)=\int\limits_{\bR^d} \int\limits_{\bR^d} G(s-t)\varphi(s) dL(t) \lambda^d(ds)\\
&\stackrel{a.s.}{=} \int\limits_{\bR^d} \int\limits_{\bR^d} G(s-t)\varphi(s)\lambda^d(ds)dL(t)
\end{align*}
(from the discussions preceeding Theorem 3.1 in [\ref{Nielsen}] it follows also that a version of $X_s$ can be chosen such that $X_s\varphi(s)$ is integrable with respect to $\lambda^d$). Further, $X:\mathcal{D}(\bR^d)\to L^0$ is clearly linear and estimates as above show that it is also continuous, hence $X$ is a generalized random process. To see that $p(D)X=q(D)\dot{L}$, observe that 
\begin{align*}
\langle X,p(D)^*\varphi\rangle&=\int\limits_{\bR^d} \int\limits_{\bR^d} G(s-t)p(D)^*\varphi(s)\lambda^d(ds)dL(t)\\
&=\int\limits_{\bR^d} (G(-\cdot)\ast p(D)^*\varphi)(t)dL(t)\\
&=\int\limits_{\bR^d} (p(D)^* G(-\cdot)\ast\varphi)(t)dL(t)\\
&=\int\limits_{\bR^d}q(D)^*\varphi(t) dL(t)\\
&=\langle \dot{L}, q(D)^* \varphi\rangle,
\end{align*}
where we used in the last equality but one that $G(-\cdot)$ is the fundamental solution of $p(-D)u=q(-D)\delta_0$. It follows that $X$ is a generalized solution of the SPDE $p(D)X=q(D)\dot{L}$.
\end{proof}

\section{Moment properties}\label{section3} We say that a generalized process $s:\mathcal{D}\to L^0(\Omega)$ has existing $\beta$-moment, $\beta>0$, if  $\bE|\langle s,\varphi\rangle|^{\beta}<\infty$ for every $\varphi\in\mathcal{D}(\bR^d)$.\\
Let $\dot{L}$ be a L\'{e}vy white noise with characteristic triplet $(a,\gamma,\nu)$. Then it is easy to see (cf. [\ref{Sato}, Theorem 25.3, p. 159]) that $\dot{L}$ has existing $\beta$-moment if and only if
\begin{align*}
\int\limits_{|z|>1} |z|^{\beta}\nu(dz)<\infty.
\end{align*}
Next we show that if $\dot{L}$ has existing $\beta$-moment then so has the CARMA generalized process given in Theorem \ref{theoremcarma}.
\begin{proposition}\label{propositionmoment}
Let $\dot{L}$ have existing $\beta$-moment ($\beta>0$) and let $p$ and $q$ be polynomials satisfying Assumption $\ref{ass1}$. Then the stationary CARMA$(p,q)$ generalized process $s$ constructed in Theorem \ref{theoremcarma} has existing $\beta$-moment, too.
\end{proposition}
\begin{proof}
Let $\varphi\in \mathcal{D}(\bR^d)$. From (\ref{constructed process}) and [\ref{Rajput}, Theorem 2.7] we see that the L\'{e}vy measure of the  random variable $s(\varphi)$ is given by
\begin{align*}
\nu_{s(\varphi)}(B)=\int\limits_{\bR^d}\int\limits_{\bR} \one_{B\setminus\{0\}}(rG\ast (1-\Delta)^{\alpha}\varphi(x))\nu(dr)\lambda^d(dx),
\end{align*}
where $G$ and $\alpha$ are defined as in the proof of Theorem $\ref{theoremcarma}$. We conclude 
\begin{align}\label{integral1}
\int\limits_{|z|>1} |z|^{\beta}\nu_{s(\varphi)}(dz)
\nonumber=&\int\limits_{\bR} |r|^{\beta}\int\limits_{|(G\ast (1-\Delta)^{\alpha}\varphi)(x)|>\frac{1}{|r|}} |G\ast(1-\Delta)^{\alpha}\varphi(x)|^{\beta}\lambda^d(dx)\nu(dr)\\
=&\int\limits_{|r|\le 1} |r|^{\beta}\int\limits_{|(G\ast (1-\Delta)^{\alpha}\varphi)(x)|>\frac{1}{|r|}} |G\ast(1-\Delta)^{\alpha}\varphi(x)|^{\beta}\lambda^d(dx)\nu(dr)\\
\nonumber&+\int\limits_{|r|>1} |r|^{\beta}\int\limits_{|(G\ast (1-\Delta)^{\alpha}\varphi)(x)|>\frac{1}{|r|}} |G\ast(1-\Delta)^{\alpha}\varphi(x)|^{\beta}\lambda^d(dx)\nu(dr).
\end{align}
For $\beta\ge 1$ we see by the Young inequality
\begin{align}\label{young}
\|G\ast(1-\Delta)^{\alpha}\varphi\|_{L^{\beta}}^{\beta}\le \|(1-\Delta)^{\alpha}\varphi\|^{\beta}_{L^{\beta}} \|G\|^{\beta}_{L^{1}}
\end{align}
and for $0<\beta<1$ we note that
\begin{align*}
&\int\limits_{\bR^d} |G\ast(1-\Delta)^{\alpha}\varphi(x)|^{\beta}\lambda^d(dx)\\
=&\int\limits_{\bR^d} |G\ast(1-\Delta)^{\alpha}\varphi(x)|^{\beta}\exp(-(b/4)\beta\|x\|)\exp((b/4)\beta\|x\|)\lambda^d(dx)\\
\le&C\left(\,\, \int\limits_{\bR^d} |G\ast(1-\Delta)^{\alpha}\varphi(x)\exp((b/4)\|x\|)| \lambda^d(dx) \right)^{\beta}\\
\le&C\left(\,\, \int\limits_{\bR^d}|(1-\Delta)^{\alpha}\varphi(y)|\exp((b/4)\|y\|)\int\limits_{\bR^d} |G(x)|\exp((b/4)\|x\|)\lambda^d(dx)\lambda^d(dy)\right)^{\beta}\\
\le&C'\|G\exp(b\|\cdot\|)\|_{L^2(\bR^d)}^{\beta}\left( \,\,\int\limits_{\bR^d}|(1-\Delta)^{\alpha}\varphi(y)|\exp((b/4)\|y\|)\lambda^d(dy)\right)^{\beta},
\end{align*}
where $b>0$ is chosen such that $\|G\exp(b\|\cdot\|)\|_{L^2}<\infty$ and $C$ and $C'$ are finite constants. 
From the previous calculations it is immediate that the term in (\ref{integral1}) corresponding to the integral when $|r|>1$ is finite for all $\beta>0$, and that the integral corresponding to the term $|r|\le 1$ is finite when $\beta\ge 2$. When $\beta\in (0,2]$ we estimate similar to (\ref{young})
\begin{align*}
&\int\limits_{|r|\le 1} |r|^{\beta}\int\limits_{|(G\ast (1-\Delta)^{\alpha}\varphi)(x)|>\frac{1}{|r|}} |G\ast(1-\Delta)^{\alpha}\varphi(x)|^{\beta}\lambda^d(dx)\nu(dr)\\
\le& \int\limits_{|r|\le 1} |r|^{2}\nu(dr)\|G\ast(1-\Delta)^{\alpha}\varphi(x)\|_{L^2(\bR^d)}^2<\infty.
\end{align*}
We conclude that $\int\limits_{|z|>1} |z|^{\beta}\nu_{s(\varphi)}(dz)$ is finite for $\beta> 0$.
\end{proof}
By the same means we obtain the following.
\begin{proposition}\label{propmoment}
Let $X$ be the mild solution of a CARMA(p,q)-equation constructed in Theorem \ref{cor1}. If the $\beta-$moment of the L\'{e}vy-white noise exists for $0< \beta\le 2$, then $\bE |X_x|^{\beta}<\infty$ for every $x\in\bR^d$.
\end{proposition}
\begin{proof}
Let $G=\mathcal{F}^{-1}\frac{q(i\cdot)}{p(i\cdot)}$ and denote the L\'{e}vy measure of $X_x=\int\limits_{\bR^d}G(x-t)dL(t)$ by $\nu_G$. Then by [\ref{Rajput}, Theorem 2.7],
\begin{align*}
\int\limits_{|z|>1}|z|^\beta \nu_G(dz)&= \int\limits_{\bR}|r|^\beta \int\limits_{|G(x)|>\frac{1}{|r|}}|G(x)|^{\beta} \lambda^d(dx)\nu(dr)\\
&=\int\limits_{|r|\le 1}|r|^\beta \int\limits_{|G(x)|>\frac{1}{|r|}}|G(x)|^{\beta} \lambda^d(dx)\nu(dr)+ \int\limits_{|r|>1}|r|^\beta \int\limits_{|G(x)|>\frac{1}{|r|}}|G(x)|^{\beta} \lambda^d(dx)\nu(dr)\\
&\le \int\limits_{|r|\le 1}|r|^2 \int\limits_{|G(x)|>\frac{1}{|r|}}|G(x)|^{2} \lambda^d(dx)\nu(dr)+ \int\limits_{|r|>1}|r|^\beta \int\limits_{|G(x)|>\frac{1}{|r|}}|G(x)|^{\beta} \lambda^d(dx)\nu(dr)\\
&=I_1+I_2.
\end{align*}
$I_1$ is clearly finite, and $I_2$ is finite since $e^{c\|\cdot\|}G\in L^2(\bR^d)$ for some $c>0$ (see the proof of Theorem \ref{cor1}) and hence $G\in L^\beta (\bR^d)$.
\end{proof}
\begin{remark}
The $\beta$ considered in Proposition \ref{propmoment} has to be smaller or equal than $2$, as otherwise there may exist some $\beta$ for which the Proposition does not hold. Look for example at the fundamental solution of the partial differential operator $\lambda-\Delta$ for some $\lambda>0$ in dimension $3$, which is given by $c\frac{\exp(-\sqrt{\kappa}\|x\|)}{\|x\|}$ with $c$ a constant. The fundamental solution does not live in $L^{3}_{loc}(\bR^3)$, see [\ref{Hofer}, Section 2.1, Equation (21)], which implies that $\bE |X_x|^3=\infty$ for all $x\in\bR^3$.
\end{remark}
As a corollary we get the following easy result.
\begin{corollary}
Let the L\'{e}vy basis $L$ have existing second moment $\sigma^2$ (i.e., $\bE(L([0,1]^d)^2)=\sigma^2$) with vanishing first moment. Then, under the assumptions of Theorem \ref{cor1}, the spectral density of the mild solution $X$ of a CARMA(p,q)-SPDE with polynomials $p$ and $q$ is given by
\begin{align}\label{eq10}
f(\xi)=\sigma^2\left|\frac{q(i\xi)}{p(i\xi)}\right|^2.
\end{align}
\end{corollary}
\begin{proof}
It is clear that $X_x$ has existing second moment and vanishing first moment. Moreover, we see from the It$\hat{\textrm{o}}$-isometry that
\begin{align*}
\bE X_0X_y=\sigma^2 \int\limits_{\bR^d} G(x)G(x-y)\lambda^d(dx).
\end{align*}
As $G$ is the inverse Fourier transform of $\frac{q(i\xi)}{p(i\xi)}$ we conclude as in [\ref{Brockwell}, Theorem 2, p. 841] that the spectral density is given by (\ref{eq10}).
\end{proof}
\section{CARMA random fields in the sense of Brockwell and Matsuda}\label{section8}
We will now analyze the CARMA random fields in the sense of Brockwell and Matsuda defined in [\ref{Brockwell}] and show that the corresponding random field defines a mild solution of a fractional stochastic partial differential equation. In our setting we find for odd dimensions the corresponding CARMA generalized processes with respect to a SPDE of type $(\ref{spde})$. A CARMA random field in the sense of Brockwell and Matsuda is defined as follows: Let $0\le q<p$, $a_{*}(z)=z^p+a_1z^{p-1}+\dotso+a_p=\prod_{i=1}^p(z-\lambda_i)$ be a polynomial with real coefficients and distinct roots $\lambda_i$ with strictly negative real parts and $b_{*}(z)=b_0+b_1z+\dotso+b_{q-1}z^{q-1}+z^q=\prod_{i=1}^q(z-\kappa_i)$ also be a polynomial with real coefficients. Assume that $\lambda_i\neq \kappa_j$ for all $i$ and $j$. Define the functions
\begin{align*}
a(z)=\prod\limits_{i=1}^p (z^2-\lambda^2_i)\textrm{ and }b(z)=\prod\limits_{i=1}^q (z^2-\kappa_i^2).
\end{align*}
Let $L$ be a L\'{e}vy basis in $\bR^d$ with finite second moment. Then the isotropic CARMA$(p,q)$ field driven by $L$ (in the sense of Brockwell and Matsuda) is given by
\begin{align}\label{CARMA}
X_t=\int\limits_{\bR^d} \sum\limits_{i=1}^p\frac{b(\lambda_i)}{a'(\lambda_i)}e^{\lambda_i ||t-u||}\,dL(u)
\end{align}
for every $t\in\bR^d$. Here, $a'$ denotes the derivative of the polynomial $a$. For a more detailed introduction see [\ref{Brockwell}, Definition 3.1, p. 837]. 
\begin{proposition}
Let $X=(X_t)_{t\in\bR^d}$ be defined by $(\ref{CARMA})$ and $d$ be odd. Then $X$ is the mild solution of the SPDE
 \begin{align}\label{spdebrockwell}
\prod_{i=1}^p a'(\lambda_i)(-\Delta+\lambda_i^2)^{\frac{d+1}{2}}X=c_d\sum\limits_{i=1}^p 2\lambda_ib(\lambda_i)\prod\limits_{j=1,j\neq i}^pa'(\lambda_j)(-\Delta+\lambda_j^2)^{\frac{d+1}{2}}dL
\end{align}
for some constant $c_d$ depending on the dimension $d$. 
\end{proposition}
\begin{proof}
We know from [\ref{Brockwell}, Theorem 2, p.841] that the Fourier transform of the isotropic CARMA kernel is given by
\begin{align*}
c_d\sum\limits_{i=1}^p\frac{2\lambda_ib(\lambda_i)}{a'(\lambda_i)(\|z\|^2+\lambda_i^2)^{\frac{d+1}{2}}}=c_d\frac{\sum\limits_{i=1}^p 2\lambda_ib(\lambda_i)\prod\limits_{j=1,j\neq i}^pa'(\lambda_j)(\|z\|^2+\lambda_j^2)^{\frac{d+1}{2}}}{\prod_{i=1}^p a'(\lambda_i)(\|z\|^2+\lambda_i^2)^{\frac{d+1}{2}}}, \,z\in\bR^d,
\end{align*}
for some constant $c_d$ dependend on the dimension $d$. We conclude that $S_d$ is the mild solution of the SPDE
\begin{align*}
\prod_{i=1}^p a'(\lambda_i)(-\Delta+\lambda_i^2)^{\frac{d+1}{2}}X=c_d\sum\limits_{i=1}^p 2\lambda_ib(\lambda_i)\prod\limits_{j=1,j\neq i}^pa'(\lambda_j)(-\Delta+\lambda_j^2)^{\frac{d+1}{2}}\dot{L}
\end{align*}
by comparing our mild solution to the definition in (\ref{mild}).
\end{proof}
For even $d$ we see that $\prod_{j=1}^p(-\Delta+\lambda_i^2)^{\frac{d+1}{2}}$ defines a fractional Laplace operator, which is defined by 
\begin{align}\label{eq999}
\prod_{j=1}^p(-\Delta+\lambda_i^2)^{\frac{d+1}{2}}\varphi:=\mathcal{F}^{-1}\prod_{j=1}^p(\sum\limits_{m=1}^d z_m^2+\lambda_j^2)^{\frac{d+1}{2}}\mathcal{F}\varphi.
\end{align} 
A fundamental solution $G$ of $Au=B\delta_0$, where $A$ and $B$ are fractional operators defined by (\ref{eq999}), is defined by $A G\ast \varphi=B\varphi$ for all $\varphi\in\mathcal{D}(\bR^d)$. Allowing this larger class of solutions we obtain the following.
\begin{proposition}
Let $X=(X_t)_{t\in\bR^d}$ be defined by $(\ref{CARMA})$. Then $X$ is the mild solution of the (fractional) SPDE
 \begin{align}\label{spdebrockwell}
\prod_{i=1}^p a'(\lambda_i)(-\Delta+\lambda_i^2)^{\frac{d+1}{2}}X=c_d\sum\limits_{i=1}^p 2\lambda_ib(\lambda_i)\prod\limits_{j=1,j\neq i}^pa'(\lambda_j)(-\Delta+\lambda_j^2)^{\frac{d+1}{2}}\dot{L}
\end{align}
for some constant $c_d$ dependend on the dimension $d$. 
\end{proposition}
\begin{proof}
Follows the same arguments as above.
\end{proof}
\section*{Acknowledgement:} Partial support by DFG grant LI 1026/6-1 is gratefully acknowledged. The author would like to thank Alexander Lindner for his patience and support and for many interesting and fruitful discussions. Moreover, the author would like to thank Claudia Kl\"uppelberg and Viet Son Pham for a valuable discussion on CARMA random fields. Last but not least the author would like to thank Paul Doukhan for pointing out the reference [\ref{S}].

\vspace{1cm}
David Berger\\
 Ulm University, Institute of Mathematical Finance, Helmholtzstra{\ss}e 18, 89081 Ulm,
Germany\\
email: david.berger@uni-ulm.de
\end{document}